%% file: main.tex
\documentclass[a4paper, 11pt]{amsart}
\usepackage[utf8]{inputenc}
\usepackage[T1]{fontenc}
\usepackage[english]{babel}
\usepackage{graphicx}
\usepackage{stmaryrd}
\usepackage{tikz}
\usepackage{tikz-cd} 
\usepackage[all]{xy}
\usepackage{comment}
\usepackage{bm}
\usepackage{comment}
\usepackage{amsmath,amsfonts,amssymb}
\usepackage[a4paper]{geometry}
\usepackage{amsthm}
\usepackage{float}
\usepackage{enumitem}
\usepackage{hyperref}
\usepackage{xcolor}
\hypersetup{
    colorlinks=true,
    linkcolor={blue},
    citecolor={blue},
    urlcolor={blue}}
\newtheorem{te}{Theorem}[section]

\newtheorem{prop}[te]{Proposition}

\newtheorem{lemme}[te]{Lemma}

\theoremstyle{definition}

\newtheorem{de}[te]{Definition}
\newtheorem{ex}[te]{Example}

\theoremstyle{remark}
\newtheorem{rque}[te]{Remark}

\newlength{\plarg}
\usetikzlibrary{cd}
\calclayout

\title{Co-Hopfian virtually free groups and elementary equivalence}
\author{Simon Andr\'e}
\date{\today}
\begin{document}

\begin{minipage}{\linewidth}

\maketitle

\vspace{5mm}

\begin{abstract}
We prove that two co-Hopfian finitely generated virtually free groups are elementarily equivalent if and only if they are isomorphic. We also prove that co-Hopfian finitely generated virtually free groups are homogeneous in the sense of model theory. 
\end{abstract}
	
\end{minipage}

\vspace{0mm}

\section{Introduction}

\thispagestyle{empty}

A group is said to be \emph{virtually free} if it has a free subgroup of finite index. In what follows, all virtually free groups are assumed to be finitely generated. A group $G$ is \emph{co-Hopfian} if every injective endomorphism of $G$ is an automorphism. This paper is concerned with the classification of co-Hopfian virtually free groups up to elementary equivalence. Notable examples of co-Hopfian virtually free groups are $\mathrm{GL}_2(\mathbb{Z})$ (which is isomorphic to the amalgamated product $D_4\ast_{D_2} D_6$ where $D_n$ denotes the dihedral group of order $2n$), and $S_{n+1}\ast_{S_n}S_{n+1}$ where $S_n$ denotes the symmetric group on $n\geq 2$ elements (see \cite{Moi13} for a characterisation of co-Hopfian groups among virtually free groups). Recall that non-abelian free groups are elementarily equivalent by the famous work of Sela \cite{Sel06}, and Kharlampovich-Myasnikov \cite{KM06}, but free groups are far from being co-Hopfian, and it is natural to expect that co-Hopfian virtually groups behave very differently from free groups from a model-theoretic point of view; it is indeed the case, as shown by the following theorem (see paragraph \ref{elem_def} for a definition of $\forall\exists$-equivalence and elementary equivalence).

\begin{te}\label{isom}Let $G$ and $G'$ be two co-Hopfian virtually free groups. The following three assertions are equivalent.
\begin{enumerate}
    \item $G$ and $G'$ are $\forall\exists$-equivalent.
    \item $G$ and $G'$ are elementarily equivalent.
    \item $G$ and $G'$ are isomorphic.
\end{enumerate}
\end{te}

It is worth pointing out that this result is not an immediate consequence of the classification of virtually free groups up to $\forall\exists$-equivalence established in \cite{And19a}. In particular, it is not true that two $\forall\exists$-equivalent virtually free groups embed into each other. For instance, $G=\mathrm{GL}_2(\mathbb{Z})\simeq D_4\ast_{D_2} D_6$ and $G'=\langle G,t \ \vert \ [t,D_2]=1\rangle$ are $\forall\exists$-equivalent but $G'$ does not embed into $G$ since $G$ is co-Hopfian and $G,G'$ are not isomorphic.

\smallskip

We also consider homogeneity. Recall that a group $G$ is \emph{homogeneous} if two tuples of elements that are indistinguishable by means of first-order formulas are in the same orbit under the action of the group of automorphisms of $G$ (see paragraph \ref{homo} for a formal definition). Perin and Sklinos \cite{PS12}, and independently Ould Houcine \cite{OH11}, proved that free groups are homogeneous (and even $\forall\exists$-homogeneous, see \ref{homo}). In \cite{And18b}, we proved that virtually free groups satisfy a weaker property, which we called almost-homogeneity. We also proved that virtually free groups are not $\forall\exists$-homogeneous in general, and conjectured that they are not homogeneous in general. However, our next result shows that co-Hopfian virtually free groups are $\forall\exists$-homogeneous.

\begin{te}\label{homogeneous}Co-Hopfian virtually free groups are $\forall\exists$-homogeneous.\end{te}

Last, we consider the class of virtually free groups $G$ that are co-Hopfian and such that $\mathrm{Out}(G)$ is finite. As an example, $\mathrm{GL}_2(\mathbb{Z})$ satisfies these two conditions. We prove the following results (see Section \ref{model} for a definition of $\exists$-equivalence, $\exists$-homogeneity and prime groups).

\begin{te}\label{isom2}Let $G$ and $G'$ be two co-Hopfian virtually free groups with finite outer automorphism groups. The following three assertions are equivalent.
\begin{enumerate}
    \item $G$ and $G'$ are $\exists$-equivalent.
    \item $G$ and $G'$ are elementarily equivalent.
    \item $G$ and $G'$ are isomorphic.
\end{enumerate}
\end{te}

\begin{te}\label{homogeneous2}Let $G$ be a co-Hopfian virtually free groups  with $\mathrm{Out}(G)$ finite. Then $G$ is $\exists$-homogeneous and prime.\end{te}


\subsection*{Acknowledgements}I thank Vincent Guirardel for useful conversations. This work was funded by the Deutsche Forschungsgemeinschaft (DFG, German Research Foundation) under Germany’s Excellence Strategy EXC 2044–390685587, Mathematics Münster: Dynamics–Geometry–Structure and by CRC 1442 Geometry: Deformations and Rigidity.

\section{Preliminaries}
\subsection{Model theory}\label{model}

For detailed background, the reader may for
instance consult \cite{Mar02}. 

\subsubsection{First-order formulas}The language of groups uses the following symbols: the quantifiers $\forall$ and $\exists$, the logical connectors $\wedge$, $\vee$, $\Rightarrow$, the equality and inequality relations $=$ and $\neq$, the symbols $1$ (standing for the identity element), ${}^{-1}$ (standing for the inverse), $\cdot$ (standing for the group multiplication), parentheses $($ and $)$, and variables $x,y,g,z\ldots$, which are to be interpreted as elements of a group. The \emph{terms} are words in the variables, their inverses, and the identity element (for instance, $x\cdot y\cdot x^{-1}\cdot y^{-1}$ is a term). For convenience, we omit group multiplication. A \emph{first-order formula} is made from terms iteratively: one can first make \emph{atomic formulas} by comparing two terms by means of the symbols $=$ and $\neq$ (for instance, $xyx^{-1}y^{-1}=1$ is an atomic formula), then one can use logical connectors and quantifiers to make new formulas from old formulas, for instance $\exists x ((x\neq 1)\wedge (\forall y (xyx^{-1}y^{-1}=1)))$. We sometimes drop parentheses when there is no ambiguity. A variable is \emph{free} if it is not bound by any quantifier $\forall$ or $\exists$. A \emph{sentence} is a formula without free variables. Given a formula $\varphi(x_1,\ldots,x_n)$, a group $G$ and a tuple $(g_1,\ldots,g_n)\in G^n$, one says that $G$ \emph{satisfies} $\varphi(g_1,\ldots,g_n)$ if this statement is true in the usual sense when the variables are interpreted as elements of $G$. An \emph{existential formula} is a formula of the form $\varphi(\bm{x}):\exists\bm{y} \ \theta(\bm{x},\bm{y})$ where $\theta(\bm{x},\bm{y})$ is a finite disjunction of conjunctions of equations and inequations in the variables of the tuples $\bm{x},\bm{y}$, i.e.\ a string of symbols of the form $\bigvee_{i=1}^{p}\bigwedge_{j=1}^{q_i}w_{i,j}(\bm{x},\bm{y})\varepsilon_i 1$, where each $\varepsilon_i$ denotes $=$ or $\neq$, $p$ and $q_i$ are integers, and $w_{i,j}$ is a reduced word in the variables of $\bm{x}$ and $\bm{y}$ and their inverses. Similarly, a \emph{$\forall\exists$-formula} is a formula of the form $\varphi(\bm{x}):\forall\bm{y}\exists\bm{z} \ \theta(\bm{x},\bm{y},\bm{z})$ where $\theta(\bm{x},\bm{y},\bm{z})$ is a finite disjunction of conjunctions of equations and inequations in the variables of the tuples $\bm{x},\bm{y},\bm{z}$. 

\subsubsection{Elementary equivalence}\label{elem_def}Two groups $G$ and $G'$ are said to be \emph{elementarily equivalent}, denoted $G\equiv G'$, if they satisfy the same first-order sentences. We say that $G$ and $G'$ are \emph{existentially equivalent}, denoted $G\equiv_{\exists} G'$, if they satisfy the same existential sentences. We define similarly the notion of \emph{$\forall\exists$-equivalence}, denoted  $\equiv_{\forall\exists}$.

\subsubsection{Homogeneity}\label{homo}Let $G$ be a group. We say that two $n$-tuples $\bm{u}$ and $\bm{v}$ of elements of $G$ have the same \emph{type} if, for every first-order formula $\phi(\bm{x})$ with $n$ free variables, $G$ satisfies $\phi(\bm{u})$ if and only if $G$ satisfies $\phi(\bm{v})$. Similarly, we say that $\bm{u}$ and $\bm{v}$ have the same existential type (respectively $\forall\exists$-type) if, for every $\exists$-formula (respectively $\forall\exists$-formula) $\phi(\bm{x})$ with $n$ free variables, $G$ satisfies $\phi(\bm{u})$ if and only if $G$ satisfies $\phi(\bm{v})$. The group $G$ is said to be \emph{homogeneous} (respectively \emph{$\exists$-homogeneous} and \emph{$\forall\exists$-homogeneous}) if for any two $n$-tuples $\bm{u}$ and $\bm{v}$ having the same type (respectively $\exists$-type and $\forall\exists$-type), there exists an automorphism $\sigma$ of $G$ mapping $\bm{u}$ to $\bm{v}$.

\subsubsection{Prime models}

A map $\varphi : G \rightarrow G'$ between two groups $G$ and $G'$ is said to be \emph{elementary} if the following condition holds: for every first-order formula $\theta(\bm{x})$ with $n$ free variables in the language of groups, and for every $n$-tuple $\bm{u}\in G^n$, $G$ satisfies $\theta(\bm{u})$ if and only if $G'$ satisfies $\theta(\bm{u})$. In particular, $\varphi$ is a morphism and is injective. The group $G$ is \emph{prime} if for every group $G'$ that is elementarily equivalent to $G$, there exists an elementary embedding $\varphi : G \rightarrow G'$. 

\subsection{Tree of cylinders}\label{tree}

Let $k\geq 1$ be an integer, let $G$ be a finitely generated group, and let $\Delta$ be a splitting of $G$ over finite groups of order $k$. Let $T$ denote the Bass-Serre tree of $\Delta$. In \cite{GL11}, Guirardel and Levitt construct a tree that only depends on the deformation space of $T$. This tree is called the tree of cylinders of $T$, denoted by $T_c$. Recall that the \emph{deformation space} of a simplicial $G$-tree $T$ is the set of $G$-trees that can be obtained from $T$ by some collapse and expansion moves, or equivalently, which have the same elliptic subgroups as $T$. We summarize below the construction of the tree of cylinders $T_c$.

First, we define an equivalence relation $\sim$ on the set of edges of $T$: we declare two edges $e$ and $e'$ to be equivalent if $G_e=G_{e'}$. Since all edge stabilizers have the same order, the union of all edges in the equivalence class of an edge $e$ is a subtree $Y_e$, called a cylinder of $T$. In other words, $Y_e$ is the subset of $T$ pointwise fixed by the edge group $G_e$. Two distinct cylinders meet in at most one point. The \emph{tree of cylinders} $T_c$ of $T$ is the bipartite tree with set of vertices $V_0(T_c)\sqcup V_1(T_c)$ such that $V_0(T_c)$ is the set of vertices $x$ of $T$ which belong to at least two cylinders, $V_1(T_c)$ is the set of cylinders $Y_e$ of $T$, and there is an edge $\varepsilon=(x, Y_e)$ between $x$ and $Y_e$ in $T_c$ if and only if $x\in Y_e$. If $Y_e$ belongs to $V_1(T_c)$, the vertex group $G_{Y_e}$ is the global stabilizer of $Y_e$ in $T$, i.e.\ the normalizer of $G_e$ in $G$ (see below). 

\begin{lemme}\label{normalisateur}The global stabilizer of $Y_e$ in $G$ coincides with $N_G(G_e)$.\end{lemme}

\begin{proof}If $g$ belongs to $\mathrm{Stab}(Y_e)$, then there exists an edge $\varepsilon \in Y_e$ such that $g e =\varepsilon$, i.e.\ $gG_eg^{-1}=G_{\varepsilon}$. In addition, $G_{\varepsilon}=G_e$ since $\varepsilon$ belongs to the same cylinder as $e$, so $gG_eg^{-1}=G_e$. Conversely, if $g$ belongs to $N_G(G_e)$, then $G_e^g=G_{g e}=G_e$, i.e.\ $g e$ and $e$ are in the same cylinder.\end{proof}

The lemma below follows immediately from the previous lemma and from the fact that a bounded subset in a tree admits a center, which is preserved by every element that preserves this bounded subset.

\begin{lemme}\label{petitlemme}Assume that $Y_e$ has bounded diameter in $T_k$. Then $N_G(G_e)$ is elliptic in $T_k$.\end{lemme}

The stabilizer of the edge $\varepsilon=(x, Y_e)$ is $G_{\varepsilon}=G_x\cap G_{Y_e}=N_{G_x}(G_e)$. Note that the inclusion $G_e \subset G_{\varepsilon}$ may be strict. As a consequence, $T$ and $T_c$ do not belong to the same deformation space in general. Note that $T_c$ may be trivial even if $T$ is not.

\subsection{An equivalence relation}
Given an element $g$ in a group $G$, we write $\mathrm{ad}(g)$ for the inner automorphism $x\mapsto g x g^{-1}$.

\begin{de}\label{equiv}Let $G$ be a non-elementary virtually free group. Let $G'$ be a group. We say that two homomorphisms $\phi, \phi': G \rightarrow G'$ are \emph{equivalent}, denoted by $\phi\sim \phi'$, if for every finite subgroup $H$ of $G$, there exists an element $g'\in G'$ such that $\phi$ and $\phi'$ coincide on $H$ up to conjugacy by $g'$, i.e.\ $\phi'_{\vert H}=\mathrm{ad}(g')\circ\phi_{\vert H}$.\end{de}

The following lemma shows that the previous equivalence relation on $\mathrm{Hom}(G,G')$ can be expressed using an existential formula.

\begin{lemme}\label{related}
Let $G$ be a finitely generated virtually free group, and let $\lbrace s_1,\ldots ,s_n\rbrace$ be a generating set of $G$. Let $G'$ be a group. There exists an existential formula $\psi_G(x_1,\ldots , x_{2n})$ with $2n$ free variables such that, for every morphisms $\phi,\phi'\in \mathrm{Hom}(G,G')$, the following assertions are equivalent:
\begin{enumerate}
\item $\phi$ and $\phi'$ are equivalent in the sense of Definition \ref{equiv};
\item $G'$ satisfies $\psi_G\left(\phi(s_1),\ldots ,\phi(s_n),\phi'(s_1),\ldots ,\phi'(s_n)\right)$.
\end{enumerate}
\end{lemme}

\begin{proof}
Let $H_1,\ldots,H_r$ be finite subgroups of $G$ such that any finite subgroup of $G$ is conjugate to some $H_i$. For every $1\leq i\leq r$, let $h_{i,1},\ldots ,h_{i,k_i}$ denote the elements of $H_i$. For every $1\leq i\leq r$ and $1\leq j\leq k_i$, there exists a word $w_{i,j}(x_1,\ldots ,x_n)$ in $n$ variables such that $h_{i,j}=w_{i,j}(s_1,\ldots ,s_n)$. Define \[\psi_G(x_1,\dots,x_{2n}):\exists y_1\ldots\exists y_{r} \bigwedge_{i=1}^{r}\bigwedge_{j=1}^{k_i}w_{i,j}(x_1,\ldots ,x_n)=y_iw_{i,j}(x_{n+1},\ldots ,x_{2n}){y_i}^{-1}.\]
Since $\phi(h_{i,j})=w_{i,j}(\phi(s_1),\ldots ,\phi(s_n))$ and $\phi'(h_{i,j})=w_{i,j}\left(\phi'(s_1),\ldots ,\phi'(s_n)\right)$ for every $1\leq i\leq r$ and $1\leq j\leq k_i$, the sentence $\psi_G\left(\phi(s_1),\ldots , \phi(s_n),\phi'(s_1),\ldots ,\phi'(s_n)\right)$ is satisfied by $G'$ if and only if the homomorphisms $\phi$ and $\phi'$ coincide up to conjugacy on every finite subgroup of $G$.\end{proof}

\begin{de}\label{rigid}Let $G$ be a non-elementary virtually free group. We say that $G$ is \emph{rigid} if every endomorphism $\phi : G\rightarrow G$ such that $\phi\sim \mathrm{id}_G$ is an automorphism.\end{de}

In Section \ref{coHrigid}, we shall prove that co-Hopfian virtually free groups are rigid.

\section{A property of virtually free groups}

A finitely generated group $G$ is virtually free if and only if it splits as a finite graph of finite groups. Such a splitting is called a \emph{Stallings splitting (or tree)} of $G$. A Stallings tree $T$ of $G$ is said to be \emph{reduced} if there is no edge of the form $e=[v,w]$ such that $G_v=G_e=G_w$ and such that $v$ and $w$ are in distinct orbits. A vertex of $T$ is called \emph{redundant} if it has degree $2$. The tree $T$ is called \emph{non-redundant} if every vertex is non-redundant. A Stallings splitting is not unique in general, but the conjugacy classes of finite vertex groups are the same in all reduced Stallings splittings of $G$. The \emph{Stallings deformation space} of $G$, denoted by $\mathcal{D}(G)$, is the set of Stallings trees of $G$ up to equivariant isometry. 

The following result is well-known, see for instance Lemmas 2.20 and 2.22 in \cite{DG11}, and Definition 2.19 in \cite{DG11} (definition of an isomorphism of graphs of groups).

\begin{prop}\label{DG}Let $T$ and $T'$ be two Stallings trees of $G$. The following two assertions are equivalent.
\begin{enumerate}
\item The quotient graphs of groups $T/G$ and $T'/G$ are isomorphic.
\item There exist an automorphism $\sigma$ of $G$ and a $\sigma$-equivariant isometry $f : T \rightarrow T'$. 
\end{enumerate}
\end{prop}

In the latter case, we use the notation $T'=T^{\sigma}$. This is not ambiguous since we consider elements in $\mathcal{D}(G)$ up to equivariant isometry. The following proposition claims that $\mathcal{D}(G)$ is cocompact under the action of $\mathrm{Aut}(G)$. We refer the reader to \cite[Proposition 2.9]{And18b}.

\begin{prop}\label{compacité}Let $G$ be a virtually free group. There exist finitely many trees $S_1,\ldots, S_n$ in $\mathcal{D}(G)$ such that, for every non-redundant tree $T\in \mathcal{D}(G)$, there exist an automorphism $\sigma$ of $G$ and an integer $1\leq \ell\leq n$ such that $T=S_{\ell}^{\sigma}$.
\end{prop}

The following proposition plays an important role in the proofs of our results. Note that when $G'$ is a torsion-free hyperbolic group and $G$ is a one-ended finitely generated group, a similar statement was proved by Sela in \cite{Sel09}. This result was generalized by Reinfeldt and Weidmann in \cite{RW14} without assuming torsion-freeness. The main point of the proposition below is that $G$ is not one-ended (except if it is finite).

\begin{prop}\label{sa}Let $G$ and $G'$ be two finitely generated virtually free goups. There exists a finite subset $F$ of $G\setminus \lbrace 1\rbrace$ such that, for every non-injective homomorphism $\phi : G\rightarrow G'$, there exists an automorphism $\sigma\in\mathrm{Aut}(G)$ such that $\ker(\phi\circ \sigma)\cap F\neq \varnothing$.
\end{prop}

\begin{proof}Let $\Delta$ and $\Delta'$ be two Stallings splittings of $G$ and $G'$ respectively. Let $T$ and $T'$ denote their Bass-Serre trees. Let $H_1,\ldots,H_r$ be finite subgroups of $G$ such that any finite subgroup of $G$ is conjugate to $H_i$ for some $1\leq i\leq r$.

Let $\phi : G\rightarrow G'$ be a non-injective homomorphism. As a first step, we build a $\phi$-equivariant map $f:T\rightarrow T'$. Let $v_1,\ldots ,v_n$ be some representatives of the orbits of vertices for the action of $G$ on the Bass-Serre tree $T$ of $\Delta$. For every $1\leq k\leq n$, $\phi(G_{v_k})$ is finite, and thus it fixes a vertex $v'_k\in T'$. Set $f(v_k)=v'_k$. Then, define $f$ on each vertex of $T$ by $\phi$-equivariance. It remains to define $f$ on the edges of $T$: if $e$ is an edge of $T$, with endpoints $v_1$ and $v_2$, there exists a unique path $e'$ from $f(v_1)$ to $f(w_2)$ in $T'$; we define $f(e)=e'$. 

If $\phi$ is not injective on the vertex groups of $T$, then $\phi$ is not injective on $H_i$ for some $1\leq i\leq r$. From now on, let us assume that $\phi$ is injective on the vertex groups of $T$.

Note that $f$ sends an edge of $T$ to a path of $T'$. Up to subdivising the edges of $T$, one can assume that $f$ sends an edge to an edge or a vertex of $T'$. Moreover, note that $f$ is not an isometry: indeed, there is a non-trivial element $g\in G$ such that $\phi(g)=1$, and hence $f(g v)=\phi(g) f(v)=f(v)$, and $g v$ is distinct from $v$, otherwise $g$ would belong to $G_v$, contradicting the assumption that $\phi$ is injective on the vertex groups of the tree $T$. As a consequence, $f$ maps an edge of $T$ to a point, or folds two edges. 

\emph{Case 1.} If $f$ maps the edge $e=[v,w]$ of $T$ to a point in $T'$, we collapse $e$ in $T$, as well as all its translates under the action of $G$. Collapsing $e$ gives rise to a new $G$-tree $T_1$ with a new vertex $x$ labelled by $G_x=\langle G_v,G_w\rangle$ if $v$ and $w$ are not in the same orbit, or $G_x=\langle G_v,g\rangle$ if $w=g v$.

\emph{Case 2.} Suppose that $f$ folds some pair of edges, as pictured below. 

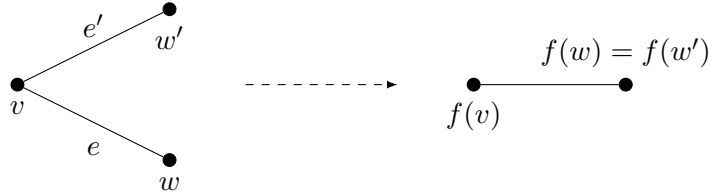
\begin{center}
\begin{tikzpicture}[scale=1]
\node[draw,circle, inner sep=1.7pt, fill, label=below:{$w$}] (A1) at (2,0) {};
\node[draw,circle, inner sep=1.7pt, fill, label=below:{${w'}$}] (A2) at (2,2) {};
\node[draw,circle, inner sep=1.7pt, fill, label=below:{$v$}] (A3) at (0,1) {};
\node[draw=none, label=below:{$e$}] (B1) at (1,0.5) {};
\node[draw=none, label=below:{$e'$}] (B2) at (1,2.2) {};
\node[draw=none, label=below:{}] (B3) at (7,2) {};
\node[draw,circle, inner sep=1.7pt, fill, label=above:{$f(w)=f(w')$}] (A4) at (8,1) {};
\node[draw,circle, inner sep=1.7pt, fill, label=below:{$f(v)$}] (A5) at (6,1) {};

\draw[-,>=latex] (A3) to (A1) ;
\draw[-,>=latex] (A3) to (A2);
\draw[-,>=latex] (A4) to (A5);
\draw[->,>=latex, dashed] (3,1) to (5,1);
\end{tikzpicture}
\end{center}
We fold $e$ and $e'$ together in $T$, as well as all their translates under the action of $G$. Folding $e$ and $e'$ gives rise to a new $G$-tree $T_1$ with a new vertex $x$ labelled by $G_x=\langle G_w,G_w'\rangle$ if $w$ and $w'$ are not in the same orbit, or $G_x=\langle G_w,g\rangle$ if ${w'}=g w$.

The map $f: T \rightarrow T'$ factors through the quotient map $\pi_1 : T \rightarrow T_1$. Let $f_1 : T_1 \rightarrow T'$ be the map such that $f = f_1\circ \pi_1$. If $T_1$ belongs the Stallings deformation space $\mathcal{D}(G)$, then the same argument as above shows that $f_1$ is not an isometry, and one can perform another collapsing or folding of edges. We get a sequence $T\rightarrow T_1\rightarrow T_2\rightarrow\cdots$. Then, observe that $T$ has only finitely many orbits of edges under the action of $G$, which implies that one can perform only finitely many collapsing or folding of edges. Hence the previous sequence of trees is necessarily finite. Let $T_{k+1}$ be the last tree in the sequence, with $k\geq 0$. Note that $T_{k+1}$ does not belong to the Stallings deformation space, otherwise one can perform one more collapsing or folding. Therefore, the last collapsing or folding in the sequence, namely $T_k\rightarrow T_{k+1}$, gives rise to an infinite vertex group. More precisely, one of the following holds, where $N$ denotes the maximal order of an element of $G'$ of finite order:
\begin{itemize}
    \item either there is an edge $[v,w]$ in $T_k$ such that $\langle G_v,G_w\rangle$ is infinite and $\phi$ kills the $N$th power of any element of $\langle G_v,G_w\rangle$ of infinite order,
    \item or there exist two edges $[v,w]$ and $[v,w']$ such that $w,w'$ are not in the same orbit, $\langle G_w,G_{w'}\rangle$ is infinite and $\phi$ kills the $N$th power of any element of $\langle G_w,G_{w'}\rangle$ of infinite order,
    \item or there exist two edges $[v,w]$ and $[v,w']$ such that $w'=gw$, $\langle G_w,g\rangle$ is infinite and $\phi$ kills the $N$th power of any element of $\langle G_w,g\rangle$ of infinite order.
\end{itemize}
Hence, one can associate to $T_k$ a finite set of elements of $G$ of infinite order such that $\phi$ kills an element of this finite set. 

Now, up to forgetting the possibly redundant vertices of $T_k$, one can assume that $T_k$ is non-redundant. By Proposition \ref{compacité}, there exist an automorphism $\sigma$ of $G$ and an integer $1\leq \ell\leq n$ such that $T_k=S_{\ell}^{\sigma}$. 

As a conclusion, one can associate to every tree $S_\ell$, with $1\leq \ell\leq n$, a finite set $F_{\ell}$ of elements of $G$ of infinite order such that for any non-injective morphism $\phi : G\rightarrow G'$, there exists $\sigma\in \mathrm{Aut}(G)$ such that $\phi\circ\sigma$ kills an element of $F_{\ell}$ for some $1\leq \ell\leq n$ or an element of $H_i$ for some $1\leq i\leq r$. Last, define $F=F_1\cup \cdots \cup F_n\cup H_1\cup \cdots \cup H_r$.\end{proof}

Before stating the next proposition, let us define a subgroup of $\mathrm{Aut}(G)$.

\begin{de}\label{Aut0}We denote by $\mathrm{Aut}_0(G)$ the subgroup of $\mathrm{Aut}(G)$ defined as follows: \[\mathrm{Aut}_0(G)=\lbrace \phi\in \mathrm{Aut}(G) \ \vert \ \phi\sim\mathrm{id}_G\rbrace.\]\end{de}

The following lemma is straightforward since a virtually free group $G$ has only finitely many conjugacy classes of finite subgroups.

\begin{lemme}\label{finite_index}The subgroup $\mathrm{Aut}_0(G)$ has finite index in $\mathrm{Aut}(G)$.\end{lemme}

Then, write $\mathrm{Aut}(G)=\sigma_1\circ \mathrm{Aut}_0(G)\cup\cdots\cup \sigma_N\circ \mathrm{Aut}_0(G)$ where $N=[\mathrm{Aut}(G):\mathrm{Aut}_0(G)]$, and observe that Proposition \ref{sa} remains true if $\mathrm{Aut}(G)$ and $F$ are replaced by $\mathrm{Aut}_0(G)$ and $\sigma_1(F)\cup\cdots\cup \sigma_N(F)$ respectively. Hence the following result follows immediately from Proposition \ref{sa}.

\begin{prop}\label{sa2}Let $G$ and $G'$ be two finitely generated virtually free goups. There exists a finite subset $F$ of $G\setminus \lbrace 1\rbrace$ such that, for every non-injective homomorphism $\phi : G\rightarrow G'$, there exists an automorphism $\sigma\in\mathrm{Aut}_0(G)$ such that $\ker(\phi\circ \sigma)\cap F\neq \varnothing$.
\end{prop}

\begin{rque}
The reason why we define this subgroup $\mathrm{Aut}_0(G)$ lies in the fact that $\phi\circ \sigma$ and $\phi$ are equivalent in the sense of Definition \ref{equiv} when $\sigma$ belongs to $\mathrm{Aut}_0(G)$. This observation will be very useful later on.
\end{rque}

\section{Co-Hopfian virtually free groups are rigid}\label{coHrigid}
\subsection{Preliminary lemmas}

\begin{lemme}\label{lemmecohopf}Let $H=\langle S \ \vert \ R \rangle$ be a group. Let $\sigma : A\rightarrow B$ be an isomorphism between two finite subgroups $A$ and $B$ of $H$. Suppose that $A$ and $B$ are conjugate in $H$. Then the HNN extension $G=H\ast_{\sigma}=\langle S,t \ \vert \ R, tat^{-1}=\sigma(a) \ \forall a\in A\rangle$ is not co-Hopfian.
\end{lemme}

\begin{proof}
Since $A=B^h$ for some $h\in H$, up to replacing $t$ by $ht$, we can assume that $\sigma$ is an automorphism of $A$. Let $m$ denote the order of $\mathrm{Aut}(A)$. We have $t^{m}at^{-m}=\sigma^m(a)=a$ for every $a\in A$, so we can define an endomorphism $\phi$ of $G$ by $\phi(h)=h$ if $h\in H$, and $\phi(t)=t^{m+1}$. A straightforward application of Britton's lemma shows that the endomorphism $\phi$ is injective, and that the stable letter $t$ has no preimage under $\phi$.\end{proof}

\begin{lemme}\label{unpetitlemme}Let $G$ be a co-Hopfian virtually free group. Let $\Delta$ be a reduced Stallings splitting of $G$. Let $k$ denote the least order of an edge group of $\Delta$. Denote by $\Delta_k$ the splitting of $G$ obtained from $\Delta$ by collapsing each edge $e$ such that $\vert G_e\vert > k$. Let $\phi$ be an endomorphism of $G$ such that $\phi\sim \mathrm{id}_G$. If $v$ is a vertex group of $\Delta_k$, then $\phi(G_v)$ is contained in a conjugate of $G_v$.\end{lemme}

\begin{proof}
Since $G_v$ does not split over a subgroup of order $k$ (by definition of $\Delta_k$), $\phi(G_v)$ fixes a vertex $v'$ in the Bass-Serre tree $T$ of $\Delta_k$. It remains to prove that $v'$ is a translate of the vertex $v$. 

First, observe that $G_v$ has a finite subgroup of order $>k$. Otherwise $G_v$ is a finite group of order $k$, and the underlying graph of $\Delta$ is a rose, and hence $\Delta$ contains the following subgraph:
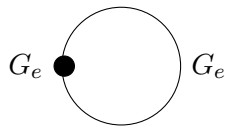
\begin{center}
\begin{tikzpicture}[scale=1]
\node[draw,circle, inner sep=2.7pt, fill, label=left:{$G_e$}] (A1) at (1.25,0) {};
\node[draw,circle, inner sep=15.7pt, label=right:{$G_e$}] (A1) at (2,0) {};
\end{tikzpicture}
\end{center}
By collapsing all but one of the edges of $\Delta$, we get a splitting of $G$ as in Lemma \ref{lemmecohopf}, which contradicts the assumption that $G$ is co-Hopfian.

Now, let $F$ be a finite subgroup of $G_v$ of order $> k$. The vertex $v$ is the unique vertex fixed by $F$ in the Bass-Serre tree $T$ of $\Delta_k$, because $\vert F\vert >k$. Moreover $\phi(F)=F^g$, thus $g  v$ is the unique vertex of $T$ fixed by $\phi(F)$. On the other hand, $\phi(G_v)$ fixes the vertex $v'$, so $\phi(F)$ fixes $v'$ as well. Since $\phi(F)$ fixes a unique point, it follows that $v'=g v$. As a consequence, $\phi(G_v)$ is contained in $G_v^g$.\end{proof}

\subsection{Characterization of co-Hopfian virtually free groups}

Let $G$ be a virtually free group and let $\Delta$ be a Stalling-Dunwoody splitting of $G$. We denote by $\Delta_k$ the splitting of $G$ obtained from $\Delta$ by collapsing each edge whose stabilizer has order $>k$. We denote by $T_k$ the Bass-Serre tree of $\Delta_k$. In his PhD thesis \cite{Moi13}, Moioli gave a complete characterization of virtually free groups that are co-Hopfian. Here below are two versions of this characterization: the first one is geometric (see Theorem \ref{Moioli1} below), and the second one is a purely group theoretical criterion expressed in terms of the normalizers of the edge groups (see Theorem \ref{Moioli2} below).

\begin{te}[Moioli]\label{Moioli1}Let $G$ be a virtually free group, and let $\Delta$ be a Stallings splitting of $G$. Then $G$ is co-Hopfian if and only if the following condition holds: for every integer $k$, and for every edge $e$ of $T_k$ such that $\vert G_e\vert =k$, the cylinder $Y_e$ of $e$ has bounded diameter in $T_k$.\end{te}

\begin{te}[Moioli]\label{Moioli2}Let $G$ be a virtually free group, and let $\Delta$ be a Stallings splitting of $G$. For every edge $e$ of $\Delta$, let $\Delta_e$ be the graph of groups obtained by collapsing each edge different from $e$ in $\Delta$. Then $G$ is co-Hopfian if and only if, for every edge $e$ of $\Delta$, the following conditions hold.
\begin{itemize}
\item[$\bullet$]If $\Delta_e$ is a splitting of the form $A\ast_C B$, then $N_A(C)=C$ or $N_B(C)=C$.
\item[$\bullet$]If $\Delta_e$ is a splitting of the form $A\ast_{\alpha}$ where $\alpha : C \rightarrow C'$ is an isomorphism between two finite subgroups of $A$, then $C$ and $C'$ are non-conjugate in $A$, and $N_A(C)=C$ or $N_A(C')=C'$.
\end{itemize}
\end{te}

\begin{rque}\label{crucial}A subgroup of a co-Hopfian group is not co-Hopfian in general. However, it follows from the previous theorem that, if $\Lambda$ is a subgraph of $\Delta$ (with the same notations as above), then the fundamental group of $\Lambda$ is co-Hopfian.\end{rque}

\subsection{Co-Hopfian virtually free groups are rigid}
 In this subsection, we shall prove that co-Hopfian virtually free groups are rigid in the sens of Definition \ref{rigid}. In other words, we shall prove that every endomorphism of a co-Hopfian virtually free group $G$ that is equivalent to $\mathrm{id}_G$ (i.e.\ that coincides with an inner automorphism on each finite subgroup of $G$) is an automorphism of $G$. First, let us prove a lemma.
 
\begin{lemme}\label{lemmeperin}Let $G$ be a co-Hopfian virtually free group. Let $\Delta$ be a Stallings splitting of $G$. Let $k$ denote the least order of an edge group of $\Delta$. Denote by $\Delta_k$ the splitting of $G$ obtained from $\Delta$ by collapsing each edge $e$ such that $\vert G_e\vert > k$. Let $\phi$ be an endomorphism of $G$ satisfying the following two properties:
\begin{enumerate}
\item $\phi\sim \mathrm{id}_G$;
\item for each vertex group $v$ of $\Delta_k$, $\phi$ is injective on $G_v$.
\end{enumerate}
Then $\phi$ is an automorphism.
\end{lemme}

\begin{proof}
Let $T_k$ denote the Bass-Serre tree of $\Delta_k$, and let $T_c$ be its tree of cylinders. Recall that the tree $T_c$ is bipartite: its set of vertices is $V_0(T_c)\sqcup V_1(T_c)$ where $V_0(T_c)$ is the set of vertices $v$ of $T_k$ that belong to at least two cylinders and $V_1(T_c)$ is the set of cylinders $Y_e$ of $T_k$. We refer the reader to subsection \ref{tree} for the exact definition of $T_c$.

As a first step, let us define a $\phi$-equivariant map $f:T_k\rightarrow T_k$. Let $v_1,\ldots ,v_n$ be some representatives of the orbits of vertices of $T_k$. By Lemma \ref{unpetitlemme}, for every $1\leq i\leq n$, there exists an element $g_i\in G$ such that $\phi(G_{v_i})\subset G_{v_i}^{g_i}$. By Remark \ref{crucial}, $G_{v_i}$ is co-Hopfian (as every vertex group of $\Delta_k$ corresponds to a subgraph of $\Delta$), and by assumption $\phi$ is injective on $G_{v_i}$, and hence $\phi(G_{v_i})= G_{v_i}^{g_i}$. Set $f(v_i)=g_i v_i$. Then, we define $f$ on every vertex of $T_k$ by $\phi$-equivariance, so that $\phi(G_v)=G_{f(v)}$ for every vertex $v$ of $T_k$. Next, we define $f$ on the edges of $T_k$ in the following way: if $e$ is an edge of $T_k$, with endpoints $v$ and $w$, there exists a unique path $e'$ from $f(v)$ to $f(w)$ in $T_k$, and we let $f(e)=e'$.

Then, the map $f$ induces a $\phi$-equivariant map $f_c : T_c \rightarrow T_c$. Indeed, for each cylinder $Y_e=\mathrm{Fix}(G_e)\subset T_k$, the image $f(Y_e)$ is contained in $\mathrm{Fix}(\varphi(G_e))$ of $T_k$, which is a cylinder since $\varphi(G_e)$ is conjugate to $G_e$. If $v\in T_k$ belongs to two cylinders, so does $f(v)$. This allows us to define $f_c$ on vertices of $T_c$, by sending $v\in V_0(T_c)$ to $f(v)\in V_0(T_c)$ and $Y\in V_1(T_c)$ to $f(Y)\in V_1(T_c)$. If $(v,Y)$ is an edge of $T_c$, then $f_c(v)$ and $f_c(Y)$ are adjacent in $T_c$.

We shall prove that $f_c$ does not fold any pair of edges and, therefore, that $f_c$ is injective. Assume towards a contradiction that there exist a vertex $v$ of $T_c$, and two distinct vertices $w$ and $w'$ adjacent to $v$ such that $f_c(w)=f_c(w')$. 

First, assume that $v$ is not a cylinder. Since $T_c$ is bipartite, $w$ and $w'$ are two cylinders, associated with two edges $e$ and $e'$ of $T_k$. Since $f_c(w)=f_c(w')$, we have $\phi(G_{e})=\phi(G_{e'})$ by definition of $f_c$. But $\phi$ is injective on $G_v$ by assumption, and $G_{e},G_{e'}$ are two distinct subgroups of $G_v$ (by definition of a cylinder). This is a contradiction.

Now, assume that $v=Y_{e}$ is a cylinder. Hence $w$ and $w'$ are two vertices of $T_k$. Since $f_c(w)=f_c(w')$, we have $f(w)=f(w')$. By definition of the map $f$, there exists an element $g\in G$ such that $w'=g w$. We have $G_{w'}=gG_wg^{-1}$ and thus $\phi(G_{w'})=\phi(g)\phi(G_w)\phi(g)^{-1}$. But $\phi(G_{w'})=G_{f(w')}=G_{f(w)}=\phi(G_w)$, and therefore $\phi(g)$ belongs to $\phi(G_w)$, so one can assume that $\phi(g)=1$ up to multiplying $g$ by an element of $G_w$. Now, observe that $\phi$ is injective on $G_v=N_G(G_e)$: indeed, by Theorem \ref{Moioli1}, $Y_e$ has bounded diameter in $T_k$, and hence $N_G(G_e)$ is elliptic in $T_k$ by Lemma \ref{petitlemme}; it follows that $\phi$ is injective on $N_G(G_e)$, as $\phi$ is injective on the vertex groups of $T_k$ by assumption. Therefore $g$ does not belong to $G_v=N_G(G_e)$ since $\phi(g)=1$. Then observe that $G_{e}$ is contained in $G_w$ and in $G_{w'}$, and that $gG_{e}g^{-1}$ is contained in $gG_wg^{-1}=G_{w'}$. The edge groups $G_{e}$ and $gG_{e}g^{-1}$ are distinct since $g$ does not lie in $N_G(G_{e})$, but $\phi(G_{e})=\phi(gG_{e}g^{-1})$ since $\phi(g)=1$. This contradicts the injectivity of $\phi$ on $G_{w'}$.

Hence, $f_c$ is injective. It follows that $\phi$ is injective. Indeed, let $g$ be an element of $G$ such that $\phi(g)=1$. Then $f_c(g v)=f_c(v)$ for each vertex $v$ of $T_c$, so $g v=v$ for each vertex $v$ of $T_c$. But $\phi$ is injective on vertex groups of $T_c$, so $g=1$.

Last, $G$ being co-Hopfian, $\phi$ is an automorphism of $G$.\end{proof}

\begin{prop}\label{theorigid}Let $G$ be a co-Hopfian virtually free group. Then $G$ is rigid in the sense of Definition \ref{rigid}: every endomorphism $\phi$ of $G$ such that $\phi\sim \mathrm{id}_G$ (i.e.\ $\phi$ coincides with a conjugation on every finite subgroup of $G$) is an automorphism.\end{prop}

\begin{proof}Let $\Delta$ be a Stallings splitting of $G$. Let $k$ denote the least order of an edge group of $\Delta$. Denote by $\Delta_k$ the splitting of $G$ obtained from $\Delta$ by collapsing each edge $e$ such that $\vert G_e\vert > k$.

Let $\phi$ be an endomorphism of $G$ such that $\phi\sim \mathrm{id}_G$. Assume towards a contradiction that $\phi$ is not injective. Then, by Lemma \ref{lemmeperin}, there exists a vertex $v$ of $\Delta_k$ such that $\phi$ is not injective on $G_v$. Moreover, by Lemma \ref{unpetitlemme}, there exists an element $g\in G$ such that $\phi(G_v)$ is contained in $G_v^g$. As a consequence, $\mathrm{ad}(g^{-1})\circ \phi$ is a non-injective endomorphism of $G_v$ that coincides with a conjugation on every finite subgroup of $G_v$.

If $G_v$ is finite, we get a contradiction since $\phi$ is injective on finite subgroups of $G$. Otherwise, the group $G_v$ splits as a non-trivial tree of finite groups $\Delta_v$, and the least order of an edge group of $\Delta_v$ is strictly greater than $k$, by definition of $\Delta$. Then, we repeat the previous operation. Since there are only finitely many orders of edge groups of $\Delta$, we get a contradiction after finitely many iterations.

Hence, every endomorphism $\phi$ of $G$ such that $\phi\sim \mathrm{id}_G$ is injective. Since $G$ is co-Hopfian by assumption, $\phi$ is an automorphism.\end{proof}

\section{Proofs of the main results}

\subsection{Elementary equivalence}

We shall prove the following result.

\begin{te}\label{isombis}Let $G$ and $G'$ be two co-Hopfian virtually free groups. The following three assertions are equivalent.
\begin{enumerate}
    \item $G$ and $G'$ are $\forall\exists$-equivalent.
    \item $G$ and $G'$ are elementarily equivalent.
    \item $G$ and $G'$ are isomorphic.
\end{enumerate}
\end{te}

This theorem is an immediate consequence of the following proposition, together with the fact that co-Hopfian virtually free groups are rigid (see Proposition \ref{theorigid}).

\begin{prop}Let $G$ and $G'$ be two virtually free groups. Suppose that $G$ is a rigid, and that $G$ and $G'$ are $\forall\exists$-equivalent. Then $G$ embeds into $G'$.\end{prop}

\begin{proof}By Proposition \ref{sa2}, there exists a finite subset $F=\lbrace g_1,\ldots,g_k\rbrace$ of $G\setminus \lbrace 1\rbrace$ such that for every non-injective homomorphism $\phi : G\rightarrow G'$, there exists an automorphism $\sigma\in\mathrm{Aut}_0(G)$ such that $\phi\circ \sigma(g_i)=1$ for some $1\leq i\leq k$. Let us fix a finite presentation $G=\langle s_1,\ldots ,s_n \ \vert \ \Sigma(s_1,\ldots,s_n)=1 \rangle$, where $\Sigma(x_1,\ldots,x_n)=1$ is a finite system of equations in the variables $x_1,\ldots,x_n$. For every integer $1\leq i\leq k$, the element $g_i$ can be written as a word $w_i(s_1,\ldots ,s_n)$ in the generators $s_1,\ldots,s_n$.

Suppose towards a contradiction that $G$ does not embed into $G'$. Then, for every morphism $\phi : G\rightarrow G'$ (which is automatically non-injective), there exists an automorphism $\sigma\in\mathrm{Aut}_0(G)$ such that $\phi\circ \sigma(g_i)=1$ for some $1\leq i\leq k$. Observe that $\phi':=\phi\circ \sigma$ and $\phi$ are equivalent in the sense of Definition \ref{equiv}, and hence $G'$ satisfies the existential formula $\psi_G\left(\phi(s_1),\ldots ,\phi(s_n),\phi'(s_1),\ldots ,\phi'(s_n)\right)$ given by Lemma \ref{related}. Then, observe that there is a one-to-one correspondence between the set of morphisms from $G$ to $G'$ and the set of solutions in $G'^n$ of the system of equations $\Sigma(x_1,\ldots,x_n)=1$. Therefore, we can write a $\forall\exists$-sentence $\mu$ (see below) that is satisfied by $G'$, and whose meaning is "for every morphism $\phi : G\rightarrow G'$, there is a morphism $\phi': G\rightarrow G'$ such that $\phi'\sim \phi$ and $\phi'(g_i)=1$ for some $1\leq i\leq k$".
\[\mu : \forall x_1 \ldots\forall x_n \exists y_1 \ldots\exists y_n \  \left(\rule{0cm}{1cm}\Sigma(x_1,\ldots,x_n)=1\Rightarrow
 \left(\rule{0cm}{1cm}
       \begin{array}{lc}
                        & \Sigma(y_1,\ldots,y_n)=1  \\
        & \wedge \ \psi_G(x_1,\ldots,x_n,y_1,\ldots,y_n) \\
                        &   \wedge\underset{1\leq i\leq k}{\bigvee} w_i(y_1,\ldots,y_n)=1 \\                                       
        \end{array}
     \right)\right).\]
This $\forall\exists$-sentence is satisfied by $G'$. But $G$ and $G'$ have the same $\forall\exists$-theory, so $\mu$ is satisfied by $G$ as well. The interpretation of $\mu$ in $G$ is "for every morphism $\phi : G\rightarrow G$, there is a morphism $\phi': G\rightarrow G$ such that $\phi'\sim \phi$ and $\phi'(g_i)=1$ for some $1\leq i\leq k$". Now, take for $\phi$ the identity of $G$. Since $G$ is assumed to be rigid, the endomorphism $\phi'$ is an automorphism, which contradicts the equality $\phi'(g_i)=1$. Therefore $G$ embeds into $G'$.\end{proof}

\subsection{Homogeneity}

In this subsection, we shall prove that co-Hopfian virtually free groups are $\forall\exists$-homogeneous. First, we need some preliminary results.

\begin{de}\label{class_permuting}Let $G$ be a virtually free group. We say that an endomorphism $\phi$ of $G$ is a class-permuting endomorphism if there exists an integer $n\geq 1$ such that $\phi^n\sim\mathrm{id}_G$ (in the sense of Definition \ref{equiv}).\end{de}

\begin{rque}The terminology is motivated by the fact that an endomorphism $\phi$ is class-permuting if and only if it induces a permutation of the set of conjugacy classes of (maximal) finite subgroups of $G$ (see Lemma \ref{terminology} below).\end{rque}

If $G$ is a co-Hopfian virtually free group, every class-permuting endomorphism of $G$ is an automorphism, by Proposition \ref{theorigid}. It is not completely obvious from Definition \ref{class_permuting} that being a class-permuting endomorphism is expressible via a first-order sentence. As a first step, we need to reformulate this definition.

\begin{lemme}\label{terminology}Let $G$ be a virtually free group. An endomorphism $\phi$ of $G$ is class-permuting if and only if the following two conditions hold.
\begin{enumerate}
\item If $A$ is a maximal finite subgroup, then $\phi(A)$ is a maximal finite subgroup.
\item If $A$ and $B$ are two maximal finite subgroups, then $\phi(A)$ and $\phi(B)$ are conjugate if and only if $A$ and $B$ are conjugate.
\end{enumerate}
\end{lemme}

\begin{proof}Let $E$ denote the set of conjugacy classes of maximal finite subgroups of $G$. Suppose that the two conditions above hold. By the first condition, $\phi$ induces a well-defined map from $E$ to $E$. By the second condition this map is injective, and hence it is bijective since $E$ is finite. Therefore, there exists an integer $m\geq 1$ such that $\phi^m$ maps every maximal finite subgroup $A$ to a conjugate of $A$. Then, there is a nonzero multiple $n$ of $m$ such that $\phi^n$ is equivalent to $\mathrm{id}_G$. Conversely, it is not hard to see that every class-permuting endomorphism of $G$ satisfies the two conditions above.\end{proof}

Using the previous lemma, we shall prove that being class-permuting can be expressed via a universal formula. We shall need a characterization of maximal finite subgroups in a virtually free group.

\begin{lemme}\label{maximaux}Let $A$ be a finite subgroup of a virtually free group $G$. The following two conditions are equivalent.
\begin{enumerate}
\item $A$ is a maximal finite subgroup of $G$.
\item For every element $g\in G\setminus A$ of finite order, there exists an element $a\in A$ such that $ga$ has infinite order.
\end{enumerate}
\end{lemme}




\begin{proof}Let $A$ be a maximal finite subgroup of $G$, and let us prove that the second condition above is satisfied. Let $\Delta$ be a reduced Stallings splitting of $G$, and let $T$ be its Bass-Serre tree. Since $A$ is finite, it is contained in a vertex group $G_v$ of $T$. Since $A$ is maximal among finite subgroups of $G$, we have $A=G_v$. If $A$ is the unique maximal finite subgroup of $G$ (i.e.\ if $G$ is a finite extension of a free group), then every element of $G$ of finite order belongs to $A$ and the second condition is obvious. Otherwise, let $g$ be an element of $G\setminus A$ of finite order. The element $g$ belongs to a vertex group $G_w$ of $T$, with $v\neq w$. Let $e$ denote the path between $v$ and $w$ in $T$. Note that $g$ does not belong to $G_e$. Moreover, since  $\Delta$ is reduced, and since $G$ is not a finite extension of a free group, there exists an element $a\in A=G_v$ such that $a$ does not belong to $G_e$. The element $ga$ has infinite order: indeed, if $ga$ had finite order, then the subgroup $\langle g,a\rangle$ would be elliptic in $T$ (by a well-known lemma of Serre), which is not possible since $\mathrm{Fix}(a)\cap\mathrm{Fix}(g)=\varnothing$.

Conversely, let us prove the contrapositive of $(2)\Rightarrow (1)$. Let $A$ be a finite subgroup of $G$ that is not maximal. Observe that $A$ is not a vertex group, because $T$ is reduced. Let $v$ be a vertex of $T$ fixed by $A$. There exists an element $g$ in $G_v\setminus A$; this element has finite order, and $ga$ has finite order for every $a\in A$.\end{proof}


The following lemma shows that being class-permuting can be expressed by means of a universal formula.

\begin{lemme}\label{permuting}
Let $G$ be a virtually free group, and let $\lbrace s_1,\ldots ,s_n\rbrace$ be a generating set for $G$. There exists a universal formula ${\theta}(x_1,\ldots , x_{n})$ with $n$ free variables such that, for every endomorphism $\phi$ of $G$, the sentence ${\theta}(\phi(s_1),\ldots ,\phi(s_n))$ is satisfied by $G$ if and only if $\phi$ is class-permuting.
\end{lemme}


\begin{proof}Let $A_1,\ldots,A_r$ be a collection on representatives of the conjugacy classes of maximal finite subgroups of $G$. By virtue of Lemma \ref{terminology}, we just have to check that the following two conditions are expressible via a universal formula:
\begin{enumerate}
\item for every $1\leq i\leq r$, $\phi(A_i)$ is a maximal finite subgroup, 
\item and for every $1\leq i\neq j\leq r$, $\phi(A_i)$ and $\phi(A_j)$ are not conjugate.
\end{enumerate}
The second condition is clearly a universal condition (in natural language: "for every $g\in G$, $\phi(A_i)$ and $g\phi(A_j)g^{-1}$ are distinct"). It remains to prove that the first condition is universal. Let $N$ denote the maximal order of a finite subgroup of $G$. By Lemma \ref{maximaux}, the first condition is equivalent to the following: for every element $g\in G\setminus \phi(A_i)$ of finite order (i.e.\ such that $g^{N!}=1$), there exists an element $h\in \phi(A_i)$ such that $gh$ has infinite order (i.e.\ such that $(gh)^{N!}\neq 1$). Again, this statement is expressible by a universal formula (indeed, the statement about the existence of $h\in \phi(A_i)$ such that $gh$ has infinite order does not require an existential quantifier since we just have to write a finite disjunction of inequalities).\end{proof}

We are ready to prove the main result of this subsection.

\begin{te}\label{homogeneous}Co-Hopfian virtually free groups are $\forall\exists$-homogeneous.\end{te}

\begin{proof}Let $G$ be a co-Hopfian virtually free group, and let $\bm{u}=(u_1,\ldots ,u_k)$ and $\bm{v}=(v_1,\ldots ,v_k)$ be two $k$-tuples of elements of $G$ having the same $\forall\exists$-type. We shall prove that there exists a class-permuting endomorphism $\phi$ of $G$ mapping $\bm{u}$ to $\bm{v}$. 

Fix a finite presentation $G=\langle s_1,\ldots ,s_n \ \vert \ \Sigma(s_1,\ldots,s_n)=1 \rangle$, where $\Sigma(x_1,\ldots,x_n)=1$ is a finite system of equations in the variables $x_1,\ldots,x_n$. For every $1\leq i\leq k$, the element $u_i$ can be written as a word $w_i(s_1,\ldots ,s_n)$.
We can write a $\exists\forall$-formula $\mu(\bm{u})$ (see below) that is satisfied by $G$, and whose meaning is "there exists a class-permuting endomorphism $\phi$ of $G$ that maps $\bm{u}$ to $\bm{u}$" (note that this statement is obviously true since we can take $\phi=\mathrm{id}_G$). In the following formula, ${\theta}(x_1,\ldots , x_{n})$ denotes the universal formula given by Lemma \ref{permuting}.
\[\mu(\bm{u}) : \exists x_1\ldots \exists x_n \ \Sigma(x_1,\ldots,x_n)=1\wedge u_i=w_i(x_1,\ldots ,x_n) \wedge \theta(x_1,\ldots ,x_n).\]Since $\bm{u}$ and $\bm{v}$ have the same $\exists\forall$-type (as they have the same $\forall\exists$-type), the formula $\mu(\bm{v})$ is satisfied by $G$ as well. Let $g_1,\ldots,g_n$ be the elements of $G$ given by the interpretation of $\mu(\bm{v})$ in $G$. We can define an endomorphism $\phi$ of $G$ mapping $s_i$ to $g_i$ for every $1\leq i\leq k$. This endomorphism maps $\bm{u}$ to $\bm{v}$ and it is class-permuting thanks to the previous lemma. By definition, there exists an integer $m\geq 1$ such that $\phi^m$ is equivalent to $\mathrm{id}_G$, and hence $\phi^m$ is an automorphism of $G$ by Proposition \ref{theorigid}. Thus $\phi$ is an automorphism of $G$.\end{proof} 

\subsection{Prime models}Recall that a group $G$ is \emph{prime} if it elementary embeds in every group $G'$ that is elementarily equivalent to $G$. In this subsection, we consider co-Hopfian virtually free groups with finite outer automorphism group. We shall see that these groups are prime and $\exists$-homogeneous. 

In \cite{Pet97}, Pettet gave a characterization of virtually free groups that have finite outer automorphism group. Note that this class is different from the class of co-Hopfian virtually free groups, as shown by the following examples.



\begin{ex}\label{cohopfoutinfini}Here is an example of a co-Hopfian virtually free group with infinitely many outer automorphisms. Let $A,B$ and $C$ be three groups isomorphic to the symmetric group $\mathcal{S}_3$. Let $a,b,c$ be elements of order 2 in $A,B,C$ respectively. Define $H=\langle a\rangle\times (B\ast_{b=c} C)$ and $G=A\ast_{\langle a\rangle} H$. In other words, $G$ is the fundamental group of the following graph of groups:
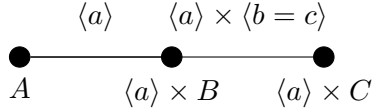
\begin{center}
\begin{tikzpicture}[scale=1]
\node[draw,circle, inner sep=2.7pt, fill, label=below:{$\langle a\rangle\times B$}] (A1) at (2,0) {};
\node[draw,circle, inner sep=2.7pt, fill, label=below:{$\langle a\rangle\times C$}] (A2) at (4,0) {};
\node[draw,circle, inner sep=2.7pt, fill, label=below:{$A$}] (A3) at (0,0) {};
\node[draw=none, label=below:{$\langle a\rangle$}] (B1) at (1,1) {};
\node[draw=none, label=below:{$\langle a\rangle\times \langle b=c\rangle$}] (B2) at (3,1) {};
\draw[-,>=latex] (A3) to (A1) ;
\draw[-,>=latex] (A3) to (A2);
\end{tikzpicture}
\end{center}
We easily see that $N_A(\langle a\rangle)=\langle a \rangle$ and $N_{\langle a\rangle\times C}(\langle a\rangle\times \langle c\rangle)=\langle a\rangle\times \langle c\rangle$, and hence $G$ is co-Hopfian by Theorem \ref{Moioli2}. On the other hand, $\mathrm{Out}(G)$ is infinite. Indeed, if $h\in H$ is an element of infinite order, the Dehn twist $\phi_h$ (defined by $\phi_h(x)=x$ if $x\in A$ and $\phi_h(x)=hxh^{-1}$ if $x\in H$) has infinite order in $\mathrm{Out}(G)$.
\end{ex}

\begin{ex}\label{outfinipascohopf}Here is an example of a virtually free group with only finitely many outer automorphisms, and which is not co-Hopfian. Let $G=\mathbb{Z}/3\mathbb{Z}\ast \mathbb{Z}/3\mathbb{Z}\simeq \mathrm{PSL}_2(\mathbb{Z})$. As a free product, $G$ is not co-Hopfian. But $\mathrm{Out}(G)$ is finite by \cite{Pet97}. More generally, Guirardel and Levitt proved in \cite{GL15} (Theorem 7.14) that a hyperbolic group $G$ has an infinite outer automorphism group if and only if $G$ splits over a $\mathcal{Z}_{\mathrm{max}}$-subgroup (i.e.\ a virtually cyclic subgroup with infinite center which is maximal for inclusion among virtually cyclic subgroups with infinite center). Therefore, if $G=A\ast_C B$ with $A$ and $B$ finite then $\mathrm{Out}(G)$ is finite. For instance, $G=\mathbb{Z}/4\mathbb{Z}\ast_{\mathbb{Z}/2\mathbb{Z}} \mathbb{Z}/6\mathbb{Z}\simeq \mathrm{SL}_2(\mathbb{Z})$ is not co-Hopfian but $\mathrm{Out}(G)$ is finite.\end{ex} 

\begin{ex}$\mathrm{GL}_2(\mathbb{Z})$ is co-Hopfian and it has only finitely many outer automorphisms.\end{ex} 

The following definition was introduced by Ould Houcine in \cite[Definiton 1.4]{OH11}.

\begin{de}
A group $G$ is said to be \emph{strongly co-Hopfian} if there exists a finite set $F\subset G\setminus \lbrace 1\rbrace$ such that for every endomorphism $\phi$ of $G$, if $\ker(\phi)\cap F=\varnothing$ then $\phi$ is an automorphism.
\end{de}

In \cite[Lemma 3.5]{OH11}, Ould Houcine observed that being strongly co-Hopfian has interesting model-theoretic consequences. 

\begin{lemme}\label{prime}
Let $G$ be a finitely presented group. If $G$ is strongly co-Hopfian, then $G$ is prime and $\exists$-homogeneous.
\end{lemme}

Examples of strongly co-Hopfian groups include torsion-free hyperbolic groups that do not split non-trivially over $\mathcal{Z}$ or as a free product (see \cite{Sel09}), $\mathrm{Out}(F_n)$, $\mathrm{Aut}(F_n)$ and the mapping-class group $\mathrm{MCG}(\Sigma_g)$ of a connected closed orientable surface of genus $g$ sufficiently large (as observed in \cite{And20}). Therefore, all these groups are prime and $\exists$-homogeneous. 

\begin{prop}Let $G$ be a co-Hopfian virtually free group with finite outer automorphism group. Then $G$ is strongly co-Hopfian. As a consequence, $G$ is prime and $\exists$-homogeneous.
\end{prop}

\begin{proof}
Let $F$ be the finite subset of $G\setminus \lbrace 1\rbrace$ given by Proposition \ref{sa}. By assumption, the group $\mathrm{Inn}(G)$ of inner automorphisms of $G$ has finite index in $\mathrm{Aut}(G)$. Write $\mathrm{Aut}(G)=\bigcup_{1\leq i\leq \ell}\sigma_i\circ \mathrm{Int}(G)$ and set $F'=\bigcup_{1\leq i\leq \ell}\sigma_i(F)$. By Proposition \ref{sa}, every endomorphism $\phi$ of $G$ such that $\ker(\phi)\cap F'= \varnothing$ is injective, and hence $\phi$ is an automorphism since $G$ is co-Hopfian.\end{proof}

\renewcommand{\refname}{Bibliography}
\bibliographystyle{alpha}
\input{main.bbl}

\vspace{8mm}

\textbf{Simon André}

Institut für Mathematische Logik und Grundlagenforschung

Westfalische Wilhelms-Universität Münster

Einsteinstraße 62

48149 Münster, Germany.

E-mail address: \href{mailto:sandre@wwu.de}{sandre@uni-muenster.de}

\end{document}

%% file: main.bbl
\def\cprime{$'$} \def\cprime{$'$}